\documentclass[reqno,10pt,centertags]{amsart}
\usepackage{amsmath,amsthm,amscd,amssymb,latexsym,esint,upref,stmaryrd,enumerate,color,verbatim,yfonts}
\usepackage{accents}
\usepackage[final]{hyperref}
\newcommand*{\mailto}[1]{\href{mailto:#1}{\nolinkurl{#1}}}
\newcommand{\arxiv}[1]{\href{http://arxiv.org/abs/#1}{arXiv:#1}}

\newcommand{\R}{{\mathbb R}}
\newcommand{\N}{{\mathbb N}}

\newcommand{\bbN}{{\mathbb{N}}}

\newcommand{\bbR}{{\mathbb{R}}}
\newcommand{\bbS}{{\mathbb{S}}}

\newcommand{\beq}{\begin{equation}}
\newcommand{\enq}{\end{equation}}

\DeclareMathOperator{\dom}{dom}

\renewcommand{\ln}{\text{\rm ln}}

\newcommand{\no}{\notag}
\newcommand{\lb}{\label}
\newcommand{\f}{\frac}

\newcommand{\ol}{\overline}
\newcommand{\bs}{\backslash}

\newcommand{\wti}{\widetilde}

\newcommand{\hatt}{\widehat}
\newcommand{\dott}{\,\cdot\,}

\newcommand{\bi}{\bibitem}

\let\geq\geqslant
\let\leq\leqslant

\makeatletter
\def\theequation{\@arabic\c@equation}

\allowdisplaybreaks
\numberwithin{equation}{section}

\newtheorem{theorem}{Theorem}[section]

\newtheorem{lemma}[theorem]{Lemma}
\newtheorem{corollary}[theorem]{Corollary}

\theoremstyle{remark}
\newtheorem{remark}[theorem]{Remark}

\begin{document}
	
\title[Optimal Power-Weighted Birman--Hardy--Rellich-type Inequalities]{Optimal Power-Weighted Birman--Hardy--Rellich-type  Inequalities on Finite Intervals and Annuli}
	
\author[F.\ Gesztesy]{Fritz Gesztesy}
\address{Department of Mathematics, Baylor University, Sid Richardson Bldg., 1410 S.\,4th Street, Waco, TX 76706, USA}
\email{\mailto{Fritz\_Gesztesy@baylor.edu}}
\urladdr{\url{https://math.artsandsciences.baylor.edu/person/fritz-gesztesy-phd}}

\author[M.\ M.\ H.\ Pang]{Michael M.\ H.\ Pang}
\address{Department of Mathematics,
University of Missouri, Columbia, MO 65211, USA}
\email{\mailto{pangm@missouri.edu}}
\urladdr{\url{https://math.missouri.edu/people/emeritus/pang}}

\dedicatory{Dedicated, with great admiration, to the memory of Vadim Tkachenko (1937--2023)}

\date{\today}
\thanks{To appear in {\it Bulletin London Math. Soc.}}
\@namedef{subjclassname@2020}{\textup{2020} Mathematics Subject Classification}

\begin{abstract}
We derive an optimal power-weighted Hardy-type inequality in integral form on finite intervals and subsequently prove the analogous inequality in differential form. We note that the optimal constant of the latter inequality differs from the former. Moreover, by iterating these inequalities we derive the sequence of power-weighted Birman--Hardy--Rellich-type inequalities in integral form on finite intervals  and then also prove the analogous sequence of inequalities in differential form. 

We use the one-dimensional Hardy-type result in differential form to derive an optimal multi-dimensional version of the power-weighted Hardy inequality in differential form on annuli (i.e., spherical shell domains), and once more employ an iteration procedure to derive the Birman--Hardy--Rellich-type sequence of power-weighted higher-order Hardy-type inequalities for annuli. 

In the limit as the annulus approaches $\bbR^n\backslash\{0\}$, we recover well-known prior results on Rellich-type inequalities on $\bbR^n\backslash\{0\}$.
\end{abstract}

\maketitle


{\scriptsize{\tableofcontents}}

\section{Introduction} \lb{s1}

We dedicate this paper to the memory of Vadim Tkachenko (1937--2023), a powerful mathematician and a wonderful human being, in the hope this paper would have put a smile on his face.

In this paper we return to optimal power-weighted Hardy inequalities, but with a twist: The interval in question is assumed to be finite and avoiding the origin. More precisely, we prove the following two optimal results ${\bf (I)}$ and ${\bf (II)}$ which appear to be new. 

Our power-weighted Hardy inequality in integral form on finite intervals in Section \ref{s2} reads as 
follows: \\[1mm] 
${\bf (I)}$ Let $\alpha \in \bbR \backslash \{1\}$ and $0 < a < b < \infty$. In addition, suppose that $\beta$ is the unique solution of 
\begin{equation}
[2/(1-\alpha)] \beta + \tan(\beta \ln(b/a)) = 0 \, \text{ in } \, 
(0,\pi/\ln(b/a)) \backslash \{\pi/[2\ln(b/a)]\}.     \lb{1.1} 
\end{equation}
Then
\begin{align} 
\begin{split} 
\int_a^b dx \, x^{\alpha} |h(x)|^2 \geq \big\{4^{-1} (1-\alpha)^2 + \beta^2\big\} \int_a^b dx \, x^{\alpha - 2} \bigg(\int_a^x dt \, |h(t)|\bigg)^2,&    \\
h \in L^2((a,b); dx).&    \lb{1.2} 
\end{split}
\end{align}
Moreover, equality holds in \eqref{1.2} for 
\begin{equation}
h_{\alpha,a,b} (x) = (x/a)^{-(\alpha+1)/2} \{[2/(1-\alpha)] \beta \, \cos(\beta \, \ln(x/a)) + \sin(\beta \, \ln(x/a))\},  \quad x \in [a,b],     \lb{1.3} 
\end{equation}
observing the fact that $h_{\alpha,a,b} \in L^2((a,b);dx)$, 
and hence the constant $4^{-1} (1-\alpha)^2 + \beta^2$ in \eqref{1.2} is optimal. 

We emphasize that the case $\alpha = 0$ in inequality \eqref{1.1}--\eqref{1.3} was recently derived by Dimitrov, Gadjev, and Ismail \cite{DGI24} (see also \cite{Ga24}), and their result inspired the writing of this paper.

Before describing the rest of our results, we recall that the Sobolev space  $W^{m, 2}_0 (\Omega)$, $m \in \bbN$, $\Omega \subseteq \bbR^n$ open, $n \in \bbN$, is defined as the completion of $C^{\infty}_0(\Omega)$ under the norm $\| \cdot \|_{m, 2, \Omega}$ given by
\begin{equation}
\|f\|_{m, 2,\Omega}^2 =  \sum _{0 \leq |\alpha| \leq m} \int_{\Omega} d^n x \, \big|\big(D^{\alpha} f\big)(x)\big|^2, \quad f \in C_0^{\infty}(\Omega),  
\end{equation}
where $\alpha = (\alpha_1,\dots,\alpha_n) \in \bbN_0^n$ denotes a multi index, $|\alpha|=\alpha_1+\cdots+\alpha_n$, and $D^{\alpha} = \partial^{|\alpha|}/\partial_{x_1}^{\alpha_1} \dots \partial_{x_n}^{\alpha_n} $ denotes the derivative of $f$ of order $\alpha$. Explicitly, $W^{m, 2}_0 (\Omega) = \ol{C_0^{\infty}(\Omega)^{\| \, \cdot \, \|_{m, 2,\Omega}}}$. 

Our power-weighted Hardy inequality in differential form on finite intervals in Section \ref{s3} reads as follows: 
\\[1mm] 
${\bf (II)}$ Let $\alpha \in \bbR$ and $0 < a < b < \infty$. Then
\begin{align}
\begin{split} 
\int_a^b dx \, x^{\alpha} |f'(x)|^2 \geq \big\{4^{-1} (1-\alpha)^2 + [\pi/\ln(b/a)]^2\big\} \int_a^b dx \, x^{\alpha-2} |f(x)|^2,&     \\
f \in W_0^{1,2}((a,b)).&    \lb{1.4} 
\end{split} 
\end{align} 
Moreover, equality holds in \eqref{1.4} for 
\begin{equation}
f_{\alpha,a,b}(x) = x^{(1-\alpha)/2} \sin ([\pi/\ln(b/a)] \ln(x/a)), \quad x \in [a,b],    \lb{1.5}
\end{equation}
observing the fact that $f_{\alpha,a,b} \in W_0^{1,2}((a,b))$, and hence the constant $4^{-1} (1-\alpha)^2 + [\pi/\ln(b/a)]^2$ in \eqref{1.4} is optimal. 

We note the interesting fact that the optimal constants in the inequalities \eqref{1.2} and \eqref{1.4} differ and refer to Remark \ref{r3.2}\,$(ii)$ for more details.

Finally, our basic multidimensional result in Section \ref{s4} is of the following form: \\[1mm]
${\bf (III)}$  Let $n \in \bbN$, $n \geq 2$, $\alpha \in \bbR$, $0 < r_1 < r_2 < \infty$, and introduce the open multi-dimensional annulus, or, spherical shell, 
\begin{equation}
A_n(r_1,r_2) = \{x \in \bbR^n \, | \, r_1 < |x| < r_2\}.     \lb{1.6} 
\end{equation}
Then one has the optimal Hardy inequality for annuli
\begin{align}  
\begin{split} 
& \int_{A_n(r_1,r_2)} d^n x \, |x|^{\alpha} |(\nabla f)(x)|^2      \\
& \quad \geq \big\{4^{-1} (2-\alpha - n)^2 + [\pi/\ln(r_2/r_1)]^2\big\}
\int_{A_n(r_1,r_2)} d^n x \, |x|^{\alpha-2} |f(x)|^2,      \\
& \hspace*{7.25cm} f \in W_0^{1,2}(A_n(r_1,r_2)).     \lb{1.7} 
\end{split} 
\end{align}
Moreover, equality holds in \eqref{1.7} for 
\begin{equation}
f_{\alpha,r_1,r_2,n}(r) = r^{(2-\alpha-n)/2} \sin([\pi/\ln(r_2/r_1)] \ln(r/r_1)), \quad r \in (r_1,r_2), 
\lb{1.8} 
\end{equation}
observing the fact that $f_{\alpha,r_1r_2,n} \in W_0^{1,2}(A_n(r_1,r_2))$, and hence the constant $4^{-1} (2-\alpha - n)^2 + [\pi/\ln(r_2/r_1)]^2$ in \eqref{1.7} is optimal. 

In all cases ${\bf (I)}$--${\bf (III)}$ we employ an iteration method to also derive power-weighted higher-order Hardy-type inequalities (also known as the power-weighted Birman--Hardy--Rellich-type sequence of inequalities). More precisely, in the context of inequalities in integral form on finite intervals we derive the following sequence of power-weighted Birman--Hardy--Rellich-type inequalities: Let $0 < a < b < \infty$, $h \in L^2((a,b);dx)$, $m \in \bbN$, and $\alpha \in \bbR \backslash \{2k-1\}_{1 \leq k \leq m}$. Then 
\begin{align}
\begin{split}
& \int_a^b dx \, x^{\alpha} |h(x)|^2 \geq \prod_{j=1}^m \big\{4^{-1}(2j-1-\alpha)^2 + \beta_{\alpha,j}^2\big\}    \\
& \quad \times \int_a^b dx \, x^{\alpha - 2m} \Bigg(\int_a^x dt_m \int_a^{t_{m}} dt_{m-1} \cdots \int_a^{t_2} dt_1 \, |h(t_1)|\Bigg)^2.    \lb{1.9} 
\end{split} 
\end{align}

The corresponding sequence of inequalities in differential form are of the following type:
Let $0 < a < b < \infty$, $\alpha \in \bbR$, and $k, m \in \bbN$, with $1 \leq k \leq m$. Then 
\begin{align}
\begin{split} 
\int_a^b dx \, x^{\alpha} |f^{(m)}(x)|^2 & \geq \prod_{j=1}^k \big\{4^{-1} (2j-1-\alpha)^2 + [\pi/\ln(b/a)]^2\big\}    \\
& \quad \times \int_a^b dx \, x^{\alpha-2k} |f^{(m-k)}(x)|^2, \quad f \in W_0^{m,2}((a,b)).    \lb{1.10} 
\end{split} 
\end{align} 

Finally, the sequence of Birman--Hardy--Rellich-type inequalities for annuli reads as follows: Let $m,n \in \bbN$, $n \geq 2$, and $\alpha \in \bbR$. Then 
\begin{align}
\begin{split} 
& \int_{A_n(r_1,r_2)} d^n x \, |x|^{\alpha} \big|\big((-\Delta)^m f\big)(x)\big|^2 \geq 
\Bigg(\prod_{j=1}^m D_{n,r_1,r_2}(\alpha - 4(j-1))\Bigg)     \\
& \quad \times \int_{A_n(r_1,r_2)} d^n x \, |x|^{\alpha-4m} |f(x)|^2, 
\quad f \in C_0^{\infty}(A_n(r_1,r_2)),
\end{split} 
\end{align}
where
\begin{align}
& D_{n,r_1,r_2}(\gamma) = \wti \beta_{n,\gamma} + 2^{-1} \big[(\gamma-2)^2 + (n-2)^2\big] [\pi/\ln(r_2/r_1)]^2 
+ [\pi/\ln(r_2/r_1)]^4,    \no \\
& \wti \beta_{n,\gamma} = \min\bigg\{\Big[\big\{\big[(\gamma-2)^2 ) - (n-2)^2\big]\big/4\big\} + j(j+n-2)\Big]^2 \, \bigg| \, j \in \bbN_0 \bigg\}; \quad \gamma \in \bbR,  
\end{align}
and 
\begin{align}
& \int_{A_n(r_1,r_2)} d^n x \, |x|^{\alpha} \big|\big(\nabla (-\Delta)^m f\big)(x)\big|^2     \no \\
& \quad \geq 
\big\{4^{-1} (\alpha+n-2)^2 + [\pi/\ln(r_2/r_1)]^2\big\} 
\Bigg(\prod_{j=1}^m D_{n,r_1,r_2}(\alpha - 2 - 4(j-1))\Bigg)    \lb {1.13} \\
& \qquad \times \int_{A_n(r_1,r_2)} d^n x \, |x|^{\alpha-2-4m} |f(x)|^2, 
\quad f \in C_0^{\infty}(A_n(r_1,r_2)).     \no 
\end{align}
For well-known special cases in the limits $r_1\downarrow 0$ and $r_2\uparrow \infty$ where the annulus $A_n(r_1,r_2)$ approaches $\bbR^n\backslash\{0\}$, we refer to Remarks \ref{r4.6} and \ref{r4.10} and Corollary \ref{c4.9}.

We cannot possibly do the enormously voluminous literature on Hardy-type inequalities any justice, and much less so in this short manuscript. Hence we just refer to the following monographs 
\cite[Sect.~1.2]{BEL15}, \cite[Ch.~1]{GM13}, \cite[p.~239--243]{HLP91}, \cite[Chs.~3--5]{KMP07}, \cite[Chs.~1--2]{KPS17}, \cite[Sects.~1.3]{Ma11}, \cite[Ch.~1]{OK90}, and to the recent papers \cite{GLMP22}, \cite{GMP22a}, \cite{GMP22}, and \cite{GPS24} for extensive bibliographies on this topic.

Finally, we briefly comment on some of the basic notation used throughout this paper.  If $T$ is a linear operator mapping (a subspace of) a Hilbert space into another, then $\dom(T)$ denotes the domain of $T$. The spectrum of a closed linear operator in a Hilbert space will be denoted by $\sigma(\cdot)$. We also use the shorthand notation $\bbN_0 = \bbN \cup \{0\}$.

\section{Optimal Power-Weighted Hardy Inequalities on Finite Intervals in Integral Form}  \lb{s2}

In this section we prove an optimal power-weighted Hardy inequality on a finite interval in integral form and subsequently derive its iterated versions, that is, finite interval higher-order power-weighted Hardy inequalities in integral form.

We start with some preliminary results.

\begin{lemma} \lb{l2.1}
Suppose that $\alpha \in \bbR$, $b \in (1, \infty)$, $L \in \big(4^{-1}(1-\alpha)^2, \infty\big)$, and 
$C = 2^{-1} \big[4L - (1-\alpha)^2\big]^{1/2}$. Then the following items $(i)$--$(iii)$ hold:\\[1mm]
$(i)$ The differential equation
\begin{equation}
- \big(x^{2-\alpha} \big(x^{\alpha} y\big)'\big)'(x) = L y(x), \quad x \in [1,b],   \lb{2.1}
\end{equation}
has the linearly independent solutions
\begin{equation}
y_1(x) = x^{- (\alpha+1)/2} \cos(C \, \ln(x)), \quad y_2(x) = x^{- (\alpha+1)/2} \sin(C \, \ln(x)), \quad x \in (1,b].      \lb{2.2} 
\end{equation}
$(ii)$ Assume in addition that 
\begin{equation} 
\alpha \in \bbR \backslash \{1\}      \lb{2.3}
\end{equation} 
and that $C_0$ is the unique solution of the equation 
\begin{equation}
[2/(1-\alpha)] C + \tan(C \, \ln(b)) =0 \, \text{ in } \, 
(0,\pi/\ln(b)) \backslash \{\pi/[2 \, \ln(b)]\}.    \lb{2.4}
\end{equation}
Then
\begin{equation}
y_0(x) = x^{-(\alpha+1)/2} \{[2/(1-\alpha)] C_0 \cos(C_0 \, \ln(x)) + \sin(C_0 \, \ln(x))\}, \quad x \in [1,b],      \lb{2.5} 
\end{equation}
is a solution of the regular boundary value problem
\begin{align}
\begin{split} 
& - \big(x^{2-\alpha} \big(x^{\alpha} y\big)'\big)'(x) = L_0 y(x), \quad 
L_0 = 4^{-1} (1-\alpha)^2 + C_0^2, \; x \in [1,b],   \\
& \, y'(1) + \alpha y(1) =0, \quad y(b) = 0.     \lb{2.6} 
\end{split} 
\end{align} 
In addition,
\begin{equation}
y_0(x) \neq 0, \quad x \in [1,b).     \lb{2.7} 
\end{equation}
$(iii)$ Assuming once more that in addition \eqref{2.3} and \eqref{2.4} hold, then 
\begin{equation}
\int_x^b ds \, s^{\alpha-2} \int_1^s dt \, y_0(t) = L_0^{-1} x^{\alpha} y_0(x), \quad x \in [1,b].  
\lb{2.8} 
\end{equation}
\end{lemma}
\begin{proof}
Since items $(i)$, $(iii)$ and most of item $(ii)$ are based on direct computations, it suffices to focus on the non-vanishing property \eqref{2.7}. We start by noting that $0 \leq C_0 \, \ln(x) < \pi$ for $x \in [1,b)$ and that $y_0(1) = C_0 \, 2/(1-\alpha) \neq 0$. Moreover, if for some $x_0 \in (1,b)$ one has $\cos(C_0 \, \ln(x_0)) = 0$, then $C_0 \, \ln(x_0) = \pi/2$ and hence 
$y_0(x_0) = x_0^{-(\alpha + 1)/2} \neq 0$. Thus it remains to consider $x \in (1,b)$ with $\cos(C_0 \, \ln(x)) \neq 0$. Suppose first that $\tan(C_0 \ln(b)) > 0$ (i.e., $\alpha \in (1,\infty)$). Then 
\begin{equation}
\tan(C_0 \, \ln(x)) = \tan(C_0 [\ln(x)/\ln(b)] \ln(b)) < \tan(C_0 \, \ln(b)) = [2/(\alpha - 1)] C_0,   \lb{2.9}
\end{equation}
and hence,
\begin{equation}
y_0(x) = x^{- (\alpha + 1)/2} \cos(C_0 \, \ln(x)) \{[2/(1-\alpha)] C_0 + \tan(C_0 \, \ln(x))\} \neq 0, \quad x \in [1,b).  
\lb{2.10}
\end{equation}
Next, suppose that $\tan(C_0 \, \ln(b)) < 0$ (i.e., $\alpha \in (-\infty,1)$). If $\pi/2 < C_0 \, \ln(x) < \pi$, then once again \eqref{2.9} and \eqref{2.10} hold. On the other hand, if $0 < C_0 \, \ln(x) < \pi/2$, then 
\begin{equation}
\tan(C_0 \, \ln(x)) > 0 > \tan(C_0 \, \ln(b)) = [2/(\alpha-1)] C_0, 
\end{equation}
and hence once more \eqref{2.10} holds.
\end{proof}

\begin{lemma} \lb{l2.2}
Suppose that $\alpha \in \bbR \backslash \{1\}$, $b \in (1, \infty)$, $L_0 = 4^{-1} (1-\alpha)^2 + C_0^2$, where $C_0$ is the unique solution of \eqref{2.4}.
Then
\begin{equation}
\int_1^b dx \, x^{\alpha} |f(x)|^2 \geq L_0 \int_1^b dx \, x^{\alpha-2} \bigg(\int_1^x dt \, |f(t)|\bigg)^2, \quad 
f \in L^2((1,b); dx).     \lb{2.12}
\end{equation}
Moreover, equality holds in \eqref{2.12} for 
\begin{equation}
f_{\alpha,b} (x) = x^{-(\alpha+1)/2} \{[2/(1-\alpha)] C_0 \, \cos(C_0 \, \ln(x)) + \sin(C_0 \, \ln(x))\}, 
\quad x \in [1,b],     \lb{2.13} 
\end{equation}
observing the fact that $f_{\alpha,b} \in L^2((1,b);dx)$, 
and hence the constant $L_0$ in \eqref{2.12} is optimal.  
\end{lemma}
\begin{proof}
Let $g \in L^2((1,b); dx)$ be real-valued and suppose that $g \neq 0$ on $(1,b)$. Then for all 
$f \in L^2((1,b); dx)$, Cauchy's inequality implies 
\begin{equation}
\bigg(\int_1^x dt \, |f(t)|\bigg)^2 \leq \int_1^x ds \, |g(s)|^2 \int_1^x dt \, |f(t)|^2\big/|g(t)|^2, 
\quad x \in (1,b).     \lb{2.14} 
\end{equation}
Thus,
\begin{align}
& \int_1^b dx \, x^{\alpha-2} \bigg(\int_1^x dt \, |f(t)|\bigg)^2 \leq \int_1^b dx \, x^{\alpha-2} 
\int_1^x ds \, g(s)^2 \int_1^x dt \, |f(t)|^2\big/g(t)^2      \no \\
& \quad = \int_1^b dt \, |f(t)|^2 g(t)^{-2} \int_t^b dx \, x^{\alpha-2} \int_1^x ds \, g(s)^2,   \lb{2.15} 
\end{align}
upon interchanging the order of integration. Introducing
\begin{equation}
G_{\alpha}(g;t) = g(t)^{-2} \int_t^b dx \, x^{\alpha-2} \int_1^x ds \, g(s)^2, \quad t \in (1,b),   \lb{2.16} 
\end{equation}
\eqref{2.15} implies
\begin{equation}
\int_1^b dx \, x^{\alpha-2} \bigg(\int_1^x dt \, |f(t)|\bigg)^2 \leq \int_1^b dt \, G_{\alpha}(g;t) |f(t)|^2, \quad 
f \in L^2((1,b); dx).   \lb{2.17} 
\end{equation}
With $y_0$ as in \eqref{2.5}, we may choose 
\begin{equation}
g(t) = |y_0(t)|^{1/2}, \quad t \in [1,b],    \lb{2.18}
\end{equation} 
and then \eqref{2.8} yields 
\begin{equation} 
G_{\alpha}(g;t) = |y_0(t)|^{-1} \int_t^b dx \, x^{\alpha-2} \int_1^x ds \, |y_0(s)| = L_0^{-1} t^{\alpha}, 
\quad t \in (1,b).   \lb{2.19}
\end{equation}
Consequently, \eqref{2.15} and \eqref{2.16} yield
\begin{equation}
L_0 \int_1^b dx \, x^{\alpha-2} \bigg(\int_1^x dt \, |f(t)|\bigg)^2 \leq \int_1^b dt \, t^{\alpha} |f(t)|^2, 
\lb{2.20} 
\end{equation}
proving \eqref{2.12}. 

Finally, choosing $f(x) = f_{\alpha,b} (x) = y_0(x)$, $x \in [1,b]$ (cf.\ \eqref{2.5} and \eqref{2.13}), then yields equality in \eqref{2.12} upon an interchange of integration and an application of \eqref{2.8}, demonstrating that $f_{\alpha,b}$ is an extremal function for \eqref{2.12}. 
\end{proof}

Appropriate scaling then yields the following result. 

\begin{theorem} \lb{t2.3}
Let $0 < a < b < \infty$ and $\alpha \in \bbR \backslash \{1\}$. In addition, suppose that $\beta_{\alpha} \in \bbR$ is the unique solution of 
\begin{equation}
[2/(1-\alpha)] \beta + \tan(\beta \ln(b/a)) = 0 \, \text{ in } \, 
(0,\pi/\ln(b/a)) \backslash \{\pi/[2\ln(b/a)]\}.     \lb{2.21} 
\end{equation}
Then
\begin{align} 
\begin{split} 
\int_a^b dx \, x^{\alpha} |h(x)|^2 \geq \big\{4^{-1} (1-\alpha)^2 + \beta_{\alpha}^2\big\} \int_a^b dx \, x^{\alpha - 2} \bigg(\int_a^x dt \, |h(t)|\bigg)^2,&    \\
h \in L^2((a,b); dx).&    \lb{2.22} 
\end{split}
\end{align}
Moreover, equality holds in \eqref{2.22} for 
\begin{align}
& h_{\alpha,a,b} (x) = (x/a)^{-(\alpha+1)/2} \{[2/(1-\alpha)] \beta_{\alpha} \, \cos(\beta_{\alpha} \, \ln(x/a)) + \sin(\beta_{\alpha} \, \ln(x/a))\},    \no \\
& \hspace*{9.5cm} x \in [a,b],      \lb{2.23} 
\end{align}
observing the fact that $h_{\alpha,a,b} \in L^2((a,b);dx)$, 
and hence the constant $4^{-1} (1-\alpha)^2 + \beta_{\alpha}^2$ in \eqref{2.22} is optimal. 
\end{theorem}
\begin{proof}
This follows from the following scaling argument: Let $h \in L^2((a,b); dx)$, and $\hatt h \in L^2((1,b/a); dx)$ be given by $\hatt h(t) = h(at)$, $t \in (1,b/a)$. Then Theorem \ref{t2.3} follows from Lemma \ref{l2.2} applied to $\hatt h$ upon noting that $h_{\alpha,a,b}(x) = f_{\alpha,b/a}(x/a)$, $x \in [a,b]$, replacing $b$ by $b/a$ in $f_{\alpha,b}$ given by \eqref{2.13}.
\end{proof}

We note that the unweighted case $\alpha =0$ in Theorem \ref{t2.3} was recently proved by  
Dimitrov, Gadjev, and Ismail \cite{DGI24}, see also \cite{Ga24}.

One can iterate inequality \eqref{2.22} to obtain higher-order power-weighted Hardy (i.e., Birman) inequalities. To describe the details we need to introduce some notation first. 

Let $h \in L^2((a,b);dx)$ and introduce $h_j: (a,b) \to [0,\infty)$, $j \in \bbN_0$, via 
\begin{align}
& h_0(x) = |h(x)|, \quad h_j(x) =\int_a^x dt \, h_{j-1}(t), \quad j \in \bbN, \; x \in (a,b).    \lb{2.24}
\end{align} 
By Cauchy's inequality and induction one concludes that 
\begin{equation}
h_j \in L^{\infty}((a,b);dx) \subset L^2((a,b);dx), \quad j \in \bbN.    \lb{2.25}
\end{equation}
In addition, for $j \in \bbN$ and $\alpha \in \bbR \backslash \{2j-1\}$, we denote by $\beta_{\alpha,j} \in \bbR$ the unique solution of 
\begin{equation}
[2/(2j-1-\alpha)] \beta + \tan(\beta \ln(b/a)) = 0 \, \text{ in } \, 
(0,\pi/\ln(b/a)) \backslash \{\pi/[2\ln(b/a)]\}.     \lb{2.26} 
\end{equation}

\begin{theorem} \lb{t2.4}
Let $0 < a < b < \infty$, $h \in L^2((a,b);dx)$, $m \in \bbN$, and $\alpha \in \bbR \backslash \{2k-1\}_{1 \leq k \leq m}$. Then 
\begin{align}
\begin{split}
& \int_a^b dx \, x^{\alpha} |h(x)|^2 \geq \prod_{j=1}^m \big\{4^{-1}(2j-1-\alpha)^2 + \beta_{\alpha,j}^2\big\}    \\
& \quad \times \int_a^b dx \, x^{\alpha - 2m} \Bigg(\int_a^x dt_m \int_a^{t_{m}} dt_{m-1} \cdots \int_a^{t_2} dt_1 \, |h(t_1)|\Bigg)^2.    \lb{2.27} 
\end{split} 
\end{align}
\end{theorem}
\begin{proof}
Inequality \eqref{2.27} is equivalent to 
\begin{align}
\int_a^b dx \, x^{\alpha} |h(x)|^2 \geq \prod_{j=1}^m \big\{4^{-1}(2j-1-\alpha)^2 + \beta_{\alpha,j}^2\big\} 
\int_a^b dx \, x^{\alpha - 2m} h_m(x)^2, \quad m \in \bbN,    \lb{2.28}
\end{align}
which we will now prove by induction on $m$. Clearly, \eqref{2.28} holds for $m=1$ by Theorem \ref{t2.3}. Next, suppose \eqref{2.28} holds for some $m_0 \in \bbN$ and that $\alpha \in \bbR\backslash\{2k-1\}_{1 \leq k \leq m_0+1}$.
Applying Theorem \ref{t2.3} to the integral on the right-hand side of \eqref{2.28} and exploiting \eqref{2.24} yields
\begin{align}
\int_a^b dx \, x^{\alpha} |h(x)|^2 &\geq \prod_{j=1}^{m_0} \big\{4^{-1}(2j-1-\alpha)^2 + \beta_{\alpha,j}^2\big\} 
\big\{4^{-1}(2m_0+1-\alpha)^2 + \beta_{\alpha,m_0+1}^2\big\}     \no \\
& \quad \times \int_a^b dx \, x^{\alpha - 2m_0 - 2} h_{m_0+1}(x)^2     \lb{2.29} \\
& \geq \prod_{j=1}^{m_0+1} \big\{4^{-1}(2j-1-\alpha)^2 + \beta_{\alpha,j}^2\big\} 
\int_a^b dx \, x^{\alpha - 2(m_0+1)} h_{m_0+1}(x)^2.    \no 
\end{align}
Thus, \eqref{2.28} holds for $m_0+1$. 
\end{proof}

\begin{remark} \lb{r2.5} 
Thus far we focused on $a$ as the lower limit of integrals of the type $\int_a^x dt$ in \eqref{2.22} and \eqref{2.27}. It is possible to put the corresponding focus on integrals of the form $\int_x^b dt$ as follows:  
For $h \in L^2((a,b);dx)$ introduce $\wti h \in L^2((a,b);dx)$ defined by
\begin{equation}
\wti h(x) = h(a+b-x) \, \text{ for a.e. $x \in (a,b)$.}     \lb{2.30} 
\end{equation}
Applying Theorem \ref{t2.3} to $\wti h$ and $\alpha \in \bbR \backslash \{1\}$ yields 
\begin{align} 
\int_a^b dx \, x^{\alpha} \big|\wti h(x)\big|^2 \geq \big\{4^{-1} (1-\alpha)^2 + \beta_{\alpha,1}^2\big\} \int_a^b dx \, x^{\alpha - 2} \bigg(\int_a^x dt \, \big|\wti h(t)\big|\bigg)^2,    \lb{2.31} 
\end{align}
and hence,
\begin{align} 
\begin{split} 
& \int_a^b dx \, (a+b-x)^{\alpha} |h(x)|^2 \geq \big\{4^{-1} (1-\alpha)^2 + \beta_{\alpha,1}^2\big\}     \\
& \quad  \times \int_a^b dx \, (a+b-x)^{\alpha - 2} \bigg(\int_x^b dt \, |h(t)|\bigg)^2, 
\quad h \in L^2((a,b); dx).    \lb{2.32} 
\end{split} 
\end{align}
Iterating \eqref{2.32} as in the proof of Theorem \ref{t2.4} then yields for $\alpha \in \bbR \backslash \{2k-1\}_{1 \leq k \leq m}$, 
\begin{align}
\begin{split}
& \int_a^b dx \, (a+b-x)^{\alpha} |h(x)|^2 \geq \prod_{j=1}^m \big\{4^{-1}(2j-1-\alpha)^2 + \beta_{\alpha,j}^2\big\}    \\
& \quad \times \int_a^b dx \, (a+b-x)^{\alpha - 2m} \Bigg(\int_x^b dt_m \int_{t_{m}}^b dt_{m-1} \cdots \int_{t_2}^b dt_1 \, |h(t_1)|\Bigg)^2.    \lb{2.33} 
\end{split} 
\end{align}
\hfill $\diamond$ 
\end{remark}

\section{Optimal Power-Weighted Hardy Inequalities on Finite Intervals in Differential Form}  \lb{s3}

In this section we prove an optimal power-weighted Hardy inequality on a finite interval in differential form and then discuss its iterations, that is, finite interval higher-order power-weighted Hardy inequalities on a finite interval. 

\begin{theorem} \lb{t3.1}
Let $0 < a < b < \infty$ and $\alpha \in \bbR$. Then
\begin{align}
\begin{split} 
\int_a^b dx \, x^{\alpha} |f'(x)|^2 \geq \big\{4^{-1} (1-\alpha)^2 + [\pi/\ln(b/a)]^2\big\} \int_a^b dx \, x^{\alpha-2} |f(x)|^2,&     \\
f \in W_0^{1,2}((a,b)).&    \lb{3.1} 
\end{split} 
\end{align} 
Moreover, equality holds in \eqref{3.1} for 
\begin{equation}
f_{\alpha,a,b}(x) = x^{(1-\alpha)/2} \sin (\pi[\ln(x/a)/\ln(b/a)]), \quad x \in [a,b],    \lb{3.2}
\end{equation}
observing the fact that $f_{\alpha,a,b} \in W_0^{1,2}((a,b))$, and hence the constant $4^{-1} (1-\alpha)^2 + [\pi/\ln(b/a)]^2$ in \eqref{3.1} is optimal. 
\end{theorem}
\begin{proof}
Since $x \mapsto x^{\alpha}$ and $x \mapsto x^{\alpha-2}$ are smooth, bounded from above, and bounded from below away from $0$ on $[a,b]$, the differential expression 
\begin{equation}
\tau_{\alpha} = - x^{2-\alpha} (d/dx) x^{\alpha} (d/dx), \quad x \in [a,b],    \lb{3.3} 
\end{equation} 
is regular on the interval $[a,b]$. Thus, the self-adjoint Dirichlet operator $T_{\alpha,D}$ associated with $\tau_{\alpha}$ (equivalently, the Friedrichs extension of the pre-minimal operator $\tau_{\alpha}|_{C_0^{\infty}((a,b))}$) in $L^2\big((a,b); x^{\alpha-2}dx\big)$ is given by
\begin{align}
& (T_{\alpha,D} f)(x) = (\tau_{\alpha} f)(x), \quad x \in [a,b],   \no \\
& f \in \dom(T_{\alpha,D}) = \big\{g \in L^2\big((a,b); x^{\alpha-2}dx\big) \, \big| \, g, g^{[1]} \in AC([a,b]);     \lb{3.4} \\
& \hspace*{3.15cm} g(a)=0=g(b); \, \tau_{\alpha} g \in L^2\big((a,b); x^{\alpha-2}dx\big)\big\}.  \no
\end{align}
Here $AC([a,b])$ represents the absolutely continuous functions on $[a,b]$, and $g^{[1]}$ represents the first quasi-derivative 
\begin{equation}
g^{[1]}(x) = x^{\alpha} g'(x), \quad x \in [a,b].    \lb{3.5} 
\end{equation}
Thus, $T_{\alpha,D}$ has purely discrete spectrum and by general principles (such as, positivity improving semigroups, or, appropriate resolvents, resp., see, e.g., \cite[Sect.~4.5]{GNZ24}, \cite[Sect.~XIII.12]{RS78}), the lowest eigenvalue $E_{0,\alpha,D}$ of $T_{\alpha,D}$ is simple with a nonvanishing eigenfunction on $(a,b)$ that one can choose strictly positive on $(a,b)$. (In particular, the eigenfunction associated with $E_{0,\alpha,D}$ is unique up to constant multiples.) By inspection, one verifies that 
\begin{equation}
E_{0,\alpha,D} = 4^{-1} (1-\alpha)^2 + [\pi/\ln(b/a)]^2    \lb{3.6} 
\end{equation}
with associated eigenfunction $f_{\alpha,a,b}$ in \eqref{3.2}, satisfying 
$f_{\alpha,a,b} > 0$ on $(a,b)$, that is,
\begin{equation}
T_{\alpha,D} f_{\alpha,a,b} = E_{0,\alpha,D} \, f_{\alpha,a,b}, \quad f_{\alpha,a,b} \in \dom(T_{\alpha,D}),     \lb{3.7} 
\end{equation} 
and hence
\begin{equation}
\inf(\sigma(T_{\alpha,D})) = E_{0,\alpha,D}, \quad T_{\alpha,D} \geq E_{0,\alpha,D} 
I_{L^2((a,b); x^{\alpha-2}dx)}.
\end{equation}
Thus, 
\begin{equation}
(\varphi,T_{\alpha,D} \varphi)_{L^2((a,b); x^{\alpha-2}dx)} \geq E_{0,\alpha,D}  
 \|\varphi\|_{L^2((a,b); x^{\alpha-2}dx)}, \quad \varphi \in C_0^{\infty}((a,b)), 
\end{equation}
and a simple integration by parts yields 
\begin{align}
\begin{split} 
\int_a^b dx \, x^{\alpha} |\varphi'(x)|^2 \geq \big\{4^{-1} (1-\alpha)^2 + [\pi/\ln(b/a)]^2\big\} \int_a^b dx \, x^{\alpha-2} |\varphi(x)|^2,&     \\
\varphi \in C_0^{\infty}((a,b)).&    \lb{3.10} 
\end{split} 
\end{align} 
Given $f \in W_0^{1,2}((a,b))$, there exists a sequence $\{f_k\}_{k \in \bbN} \subset C_0^{\infty}((a,b))$ such that $f_k \underset{k \to \infty}{\longrightarrow} f$ in $W_0^{1,2}((a,b))$, hence,
\begin{align} 
\begin{split}
& \lim_{k \to \infty} \int_a^b dx \, |f_k(x)|^2 = \int_a^b dx \, | f(x)|^2,    \\
& \lim_{k \to \infty} \int_a^b dx \, |f_k'(x)|^2 = \int_a^b dx \, |f'(x)|^2.    \lb{3.11}
\end{split} 
\end{align}
Employing once again that $x \mapsto x^{\alpha}$ and $x \mapsto x^{\alpha-2}$ are bounded from above and bounded from below away from $0$, \eqref{3.11} implies  
\begin{align} 
\begin{split}
& \lim_{k \to \infty} \int_a^b dx \, x^{\alpha-2} |f_k(x)|^2 
= \int_a^b dx \, x^{\alpha-2} |f(x)|^2,    \\
& \lim_{k \to \infty} \int_a^b dx \, x^{\alpha} |f_k'(x)|^2 
= \int_a^b dx \, x^{\alpha} |f'(x)|^2.   \lb{3.12}
\end{split} 
\end{align}
Thus, \eqref{3.1} follows from \eqref{3.10} and \eqref{3.12}.   
\end{proof} 

\begin{remark} \lb{r3.2} 
$(i)$ When either $a \downarrow 0$ or $b \uparrow \infty$, or both $a \downarrow 0$ and $b \uparrow \infty$, the constant in the right side of \eqref{3.1} tends to $4^{-1}(1 - \alpha)^2$, which is known to be sharp for intervals of the form $(0, b)$ (cf.\ \cite[Thm.~4.1]{GMP22}) or $(a, \infty)$ (cf.\ \cite[Thm.~4.7]{GMP22}) or $(0, \infty)$ (see, e.g., \cite[Lemma~2.1]{GMP22a}). One also notes that $f_{\alpha, a, b}(x) \rightarrow 0$ when either $a \downarrow 0$ or $b \uparrow \infty$ or both $a \downarrow 0$ and $b \uparrow \infty$. \\[1mm] 
$(ii)$ Since $C_0^{\infty}((a,b))$ is dense in $L^2((a,b);dx)$, inequality \eqref{2.22} holds for $f \in C_0^{\infty}((a,b))$ and the constant 
\begin{equation}
K = 4^{-1}(1-\alpha)^2 + \beta^2    \lb{3.13}
\end{equation}
in \eqref{2.22} is optimal for $f \in C_0^{\infty}((a,b))$. When comparing $K$ with the constant 
\begin{equation}
M = 4^{-1} (1-\alpha)^2 + [\pi/\ln(b/a)]^2
\end{equation}
in inequality \eqref{3.1} one obtains
\begin{equation}
K < M.
\end{equation}
This is no contradiction since when comparing the left-hand sides in inequalities \eqref{2.22} for the function space $C_0^{\infty}((a,b))$ and \eqref{3.1}, the functions range over $C_0^{\infty}((a,b))$ and over
\begin{equation}
\wti C_0^{\infty}((a,b)) = \{g' \in C_0^{\infty}((a,b)) \, | \, g \in C_0^{\infty}((a,b))\} \subset C_0^{\infty}((a,b)),
\end{equation}
respectively. Since $K < M$, this implies that $\wti C_0^{\infty}((a,b))$ is not dense in $C_0^{\infty}((a,b))$ with respect to the norm in $L^2((a,b);dx)$. \\[1mm]
$(iii)$ In view of our recent papers, \cite{GLMP22}, \cite{GMP22}, \cite{GPS24}, one observes that it is not possible to improve \eqref{3.1} with log-refinement terms, or else, with any improvement term which is continuous and positive on $(a,b)$. This follows from the existence of an extremal function $f_{\alpha,a,b}$ in \eqref{3.2}, which is continuous and nonzero on $(a,b)$; any such improvement term which is continuous and positive would yield a contradiction. \\[1mm]
$(iv)$ For relationships between the integral form (cf.\ \eqref{2.22}) and the differential form (cf.\ \eqref{3.1}) of Hardy-type inequalities with general weights, see, for example \cite[Lemma~1.10, Remark~1.11]{OK90}. \hfill $\diamond$ 
\end{remark}

Next, we also iterate \eqref{3.1} to obtain higher-order power-weighted Hardy (i.e., Birman) inequalities in differential form.

\begin{theorem} \lb{t3.3} 
Let $0 < a < b < \infty$, $\alpha \in \bbR$, and $k, m \in \bbN$, with $1 \leq k \leq m$. Then 
\begin{align}
\begin{split} 
\int_a^b dx \, x^{\alpha} |f^{(m)}(x)|^2 & \geq \prod_{j=1}^k \big\{4^{-1} (2j-1-\alpha)^2 + [\pi/\ln(b/a)]^2\big\}    \\
& \quad \times \int_a^b dx \, x^{\alpha-2k} |f^{(m-k)}(x)|^2, \quad f \in W_0^{m,2}((a,b)).    \lb{3.17} 
\end{split} 
\end{align} 
\end{theorem}
\begin{proof}
Since $C_0^{\infty}((a,b))$ is dense in $W_0^{m,2}((a,b))$ and the functions $x \mapsto x^{\alpha}$ and $x \mapsto x^{\alpha - 2k}$ are bounded above and bounded below away from $0$ on $[a,b]$, it suffices again to prove \eqref{3.17} for $f \in C_0^{\infty}((a,b))$. In addition, it suffices to consider \eqref{3.17} for $k=m$, that is, it suffices to prove
\begin{align}
\begin{split} 
\int_a^b dx \, x^{\alpha} |f^{(m)}(x)|^2 & \geq \prod_{j=1}^m \big\{4^{-1} (2j-1-\alpha)^2 + [\pi/\ln(b/a)]^2\big\}    \\
& \quad \times \int_a^b dx \, x^{\alpha-2m} |f(x)|^2, \quad f \in C_0^{\infty}((a,b)).    \lb{3.18} 
\end{split} 
\end{align} 
Once again we will prove \eqref{3.18} by induction on $m$. Clearly, \eqref{3.18} holds for $m=1$ by Theorem \ref{t3.1}. Next, suppose that \eqref{3.18} holds for some $m_0 \in \bbN$. Applying \eqref{3.18} to $f'$ yields
\begin{align}
\begin{split} 
\int_a^b dx \, x^{\alpha} |f^{(m_0+1)}(x)|^2 & \geq \prod_{j=1}^{m_0} \big\{4^{-1} (2j-1-\alpha)^2 + [\pi/\ln(b/a)]^2\big\}    \\
& \quad \times \int_a^b dx \, x^{\alpha-2m_0} |f'(x)|^2, \quad f \in C_0^{\infty}((a,b)).    \lb{3.19} 
\end{split} 
\end{align} 
Applying Theorem \ref{t3.1} to the integral on the right-hand side of \eqref{3.19} results in
\begin{align}
& \int_a^b dx \, x^{\alpha} |f^{(m_0+1)}(x)|^2 \geq \prod_{j=1}^{m_0} \big\{4^{-1} (2j-1-\alpha)^2 + [\pi/\ln(b/a)]^2\big\}   
\no \\
& \qquad \times \big\{4^{-1} (2m_0+1-\alpha)^2 + [\pi/\ln(b/a)]^2\big\} 
\int_a^b dx \, x^{\alpha-2m_0-2} |f(x)|^2,    \no \\
& \quad = \prod_{j=1}^{m_0+1} \big\{4^{-1} (2j-1-\alpha)^2 + [\pi/\ln(b/a)]^2\big\}    \lb{3.20} \\
& \qquad \times \int_a^b dx \, x^{\alpha-2(m_0+1)} |f(x)|^2, \quad f \in C_0^{\infty}((a,b)).    \no 
\end{align} 
and hence in \eqref{3.18} for $m_0+1$. 
\end{proof}

\section{Power-Weighted Hardy Inequalities on Multi-Dimensional Annuli}  \lb{s4}

In this section we turn to a multi-dimensional application and derive an optimal power-weighted Hardy inequality in differential form on annuli (i.e., spherical shell domains), and subsequently use an iteration procedure to derive the Birman--Hardy--Rellich-type sequence of power-weighted higher-order Hardy-type inequalities on annuli. 

In the limit as the annulus approaches $\bbR^n \backslash \{0\}$, we recover well-known prior results on Rellich-type inequalities. 

Let $n \in \bbN$, $n \geq 2$, $\alpha \in \bbR$, $0 < r_1 < r_2 < \infty$, and introduce the open multi-dimensional annulus, or, spherical shell, 
\begin{equation}
A_n(r_1,r_2) = \{x \in \bbR^n \, | \, r_1 < |x| < r_2\}.     \lb{4.1} 
\end{equation}
Moreover, we denote by $\bbS^{n-1}$ the $(n-1)$-dimensional unit sphere in $\bbR^n$, by $d^{n-1} \omega$ the usual volume measure on $\bbS^{n-1}$, by $\nabla_{\bbS^{n-1}}$ the gradient operator on $C_0^{\infty}(\bbS^{n-1})$, and denote spherical coordinates by $\bbR^n\backslash \{0\} \ni x = (r, \theta) \in (0,\infty) \times \bbS^{n-1}$. Then
\begin{align} 
\begin{split} 
|(\nabla f)(x)|^2 = |(\partial f)/(\partial r) (r,\theta)|^2 + r^{-2} |(\nabla_{\bbS^{n-1}} f(r, \dott))(\theta)|^2,&    \lb{4.2} \\
x = (r,\theta) \in \bbR^n \backslash \{0\}, \; f \in C_0^{\infty}(\bbR^n \backslash \{0\}).&  
\end{split} 
\end{align}

Then the optimal power-weighted Hardy inequality on the annulus $A_n(r_1,r_2)$ reads as follows.

\begin{theorem} \lb{t4.1} 
One has
\begin{align}  
\begin{split} 
& \int_{A_n(r_1,r_2)} d^n x \, |x|^{\alpha} |(\nabla f)(x)|^2      \\
& \quad \geq \big\{4^{-1} (2-\alpha - n)^2 + [\pi/\ln(r_2/r_1)]^2\big\}
\int_{A_n(r_1,r_2)} d^n x \, |x|^{\alpha-2} |f(x)|^2,      \\
& \hspace*{7.25cm} f \in W_0^{1,2}(A_n(r_1,r_2)).     \lb{4.3} 
\end{split} 
\end{align}
Moreover, equality holds in \eqref{4.3} for 
\begin{equation}
f_{\alpha,r_1,r_2,n}(r) = r^{(2-\alpha-n)/2} \sin(\pi [\ln(r/r_1)/\ln(r_2/r_1)]), \quad r \in (r_1,r_2),     \lb{4.4} 
\end{equation}
observing the fact that $f_{\alpha,r_1r_2,n} \in W_0^{1,2}(A_n(r_1,r_2))$, and hence the constant $4^{-1} (2-\alpha - n)^2 + [\pi/\ln(r_2/r_1)]^2$ in \eqref{4.3} is optimal. 
\end{theorem}
\begin{proof}
By Theorem \ref{t3.1} and \eqref{4.2} one infers that 
\begin{align}
& \int_{A_n(r_1,r_2)} d^n x \, |x|^{\alpha} |(\nabla f)(x)|^2      \no \\
& \quad = \int_{\bbS^{n-1}} d^{n-1} \omega(\theta) \int_{r_1}^{r_2} r^{n-1} dr \, 
r^{\alpha} \big\{|(\partial f)/(\partial r)(r,\theta)|^2 
+ r^{-2} |(\nabla_{\bbS^{n-1}}f(r,\dott))(\theta)|^2\big\}    \no \\
& \quad \geq \int_{\bbS^{n-1}} d^{n-1} \omega(\theta) \int_{r_1}^{r_2} dr \, r^{\alpha + n - 1} 
|(\partial f)/(\partial r)(r,\theta)|^2     \no \\
& \quad \geq \big[4^{-1} (2 -\alpha - n)^2 + [\pi/\ln(r_2/r_1)]^2\big]
\int_{\bbS^{n-1}} d^{n-1} \omega(\theta) \int_{r_1}^{r_2} r^{\alpha+n-3} dr \, |f(r,\theta)|^2     \no \\
& \quad = \big[4^{-1} (2 -\alpha - n)^2 + [\pi/\ln(r_2/r_1)]^2\big] 
\int_{A_n(r_1,r_2)} d^n x \, |x|^{\alpha-2} \, |f(x)|^2,     \lb{4.5} \\
& \hspace*{7.1cm} f \in C_0^{\infty}(A_n(r_1,r_2)).      \no
\end{align}
Since $C_0^{\infty}(A_n(r_1,r_2))$ is dense in $W_0^{1,2}(A_n(r_1,r_2))$, and once more utilizing the by now familiar argument that $x \mapsto |x|^{\alpha-2}$ and $x \mapsto |x|^{\alpha}$ are bounded from above and from below away from $0$, \eqref{4.3} follows for all 
$f \in W_0^{1,2}(A_n(r_1,r_2))$ from \eqref{4.5} as in the proof of Theorem \ref{t3.1}. 

Noting that $f_{\alpha,r_1,r_2,n}(r) = f_{\alpha+n-1,r_1,r_2}(r)$, $r \in [r_1,r_2]$, where 
$f_{\alpha+n-1,r_1,r_2}$ is given by \eqref{3.2} with $\alpha,a,b$ replaced by $\alpha+n-1,r_1,r_2$, respectively, one infers that $f_{\alpha,r_1,r_2,n} \in W_0^{1,2}(A_n(r_1,r_2))$ and that equality holds in \eqref{4.3} for $f = f_{\alpha,r_1,r_2,n}$.
\end{proof}

For Hardy-type inequalities on spherical shell domains of a rather different nature we also refer to \cite[Lemma~1]{KW72}.

In order to prove higher-order (i.e., Rellich, etc.) Hardy-type inequalities we need some preparations. 

We denote by $- \Delta_{\bbS^{n-1}}$, $n \in \bbN$, $n\geq 2$, (minus) the Laplace--Beltrami operator in $L^2(\bbS^{n-1}; d^{n-1}\omega)$, with $d^{n-1}\omega$ the standard surface measure on $\bbS^{n-1}$. One recalls that the spectrum of $- \Delta_{\bbS^{n-1}}$, consisting of discrete eigenvalues $\lambda_j$, $j\in\bbN_0$, only, is given by 
\begin{equation}
\sigma(- \Delta_{\bbS^{n-1}}) = \{\lambda_j\}_{j \in \bbN_0} = \{j(j+n-2)\}_{j \in \bbN_0},
\end{equation}
with multiplicities $m(\lambda_j)$, $j\in\bbN_0$, given by 
\begin{equation}
m(\lambda_j) = \f{2j+n-2}{j+n-2}\binom{j+n-2}{n-2}, \quad j \in \bbN_0,
\end{equation}
with corresponding normalized eigenfunctions
\begin{equation}
\varphi_{j,\ell}, \quad \|\varphi_{j,\ell}\|_{L^2(\bbS^{n-1}; d^{n-1}\omega)} = 1, \quad \ell \in \{1,\dots,m(\lambda_j)\}, \; j \in \bbN_0.
\end{equation}
In particular, the collection of all eigenfunctions, $\{\varphi_{j,\ell}\}_{1\leq \ell \leq m(\lambda_j), \, j \in \bbN_0}$, represents an orthonormal basis of $L^2(\bbS^{n-1}; d^{n-1}\omega)$.

Introducing 
\begin{align}\lb{4.9}
\begin{split}
F_{f,j,\ell}(r) = (\varphi_{j,\ell}, f(r, \cdot) )_{L^{2}(\bbS^{n-1},d^{n-1}\omega)} = \int_{\bbS^{n-1}} d^{n-1}\omega(\theta) \, 
\ol{\varphi_{j,\ell}(\theta)} f(r,\theta),& \\
f \in C_{0}^{\infty}(\R^{n} \bs \{0\}), \; r \in (0,\infty), \; \ell \in \{1,\dots,m(\lambda_j)\}, \; j \in \N_{0},&
\end{split}
\end{align}
we start by recalling the following results from \cite{GMP24}:

\begin{lemma} [\cite{GMP24}, Lemma~2.1\,$(i)$, $(iv)$] \lb{l4.2}
Let $f \in C_{0}^{\infty}(A_n(r_1,r_2))$. Then the following items $(i)$ and $(ii)$ hold: \\[1mm] 
$(i)$ $F_{f,j,\ell} \in C_{0}^{\infty}((0,\infty))$ for all $\ell \in \{1,\dots,m(\lambda_j)\}$, $j \in \N_{0}$. \\[1mm]
$(ii)$ One has, in the sense of $L^{2}(\R^n;d^nx)$, 
\begin{align}\lb{4.10}
\begin{split} 
(-\Delta f )(x) = \sum_{j\in\bbN_0}\sum_{\ell=1}^{m(\lambda_j)} \big[-r^{1-n} \big( r^{n-1} F_{f,j,\ell}^{\, \prime}(r)\big)' + \lambda_{j} r^{-2} F_{f,j,\ell}(r)\big] \varphi_{j,\ell}(\theta),&   \\
 x =(r,\theta) \in A_n(r_1,r_2).&  
\end{split}
\end{align}
\end{lemma}

\begin{lemma} [\cite{GMP24}, Lemma~2.3] \lb{l4.3}
For all $\alpha \in \bbR$, $f \in C_{0}^{\infty}((r_1,r_2))$, $f$ real-valued, and $\lambda \in [0,\infty)$, one infers that 
\begin{align}\lb{4.11}
&\int_{r_1}^{r_2} dr \, r^{\alpha+n-1} \big[-r^{1-n} \big( r^{n-1} f'(r)\big)' + \lambda r^{-2} f(r)  \big]^{2}     \no \\
&\quad= \int_{r_1}^{r_2} dr \, r^{\alpha+n-1} |f''(r)|^{2} + [2\lambda + (n-1)(1-\alpha)]\int_{r_1}^{r_2} dr \, r^{\alpha+n-3} |f'(r)|^{2}     \no \\
&\qquad + \lambda[\lambda + (\alpha+n-4)(2-\alpha)] \int_{r_1}^{r_2} dr \, r^{\alpha+n-5} |f(r)|^{2}.
\end{align}
\end{lemma}

\begin{lemma} \lb{l4.4}
For all $\alpha \in \bbR$, $f \in C_{0}^{\infty}((r_1,r_2))$, $f$ real-valued, and $\lambda \in [0,\infty)$, one obtains  
\begin{align}
& \int_{r_1}^{r_2} dr \, r^{\alpha+n-1} \big|- r^{1-n} \big(r^{n-1} f'(r)\big)' + \lambda r^{-2} f(r)\big|^2  
\no \\
& \quad \geq \Big\{\big(\wti \gamma_{n,\alpha} + \lambda\big)^2 +[\pi/\ln(r_2/r_1)]^4     \\
& \qquad \;\;\, + \big(2^{-1}\big[(\alpha-2)^2 + (n-2)^2\big] + 2\lambda\big)[\pi/\ln(r_2/r_1)]^2 \Big\}  
\int_{r_1}^{r_2} dr \, r^{\alpha +n -5} |f(r)|^2,     \no 
\end{align}
where 
\begin{equation}
\wti \gamma_{n,\alpha} = \big[(n-2)^2 - (\alpha-2)^2\big]\big/4, \quad \alpha \in \bbR.    \lb{4.13} 
\end{equation}
\end{lemma}
\begin{proof}
Applying Lemma \ref{l4.3}, Theorem \ref{t3.1}, and Theorem \ref{t3.3}, one obtains 
\begin{align}
& \int_{r_1}^{r_2} dr \, r^{\alpha+n-1} \big|- r^{1-n} \big(r^{n-1} f'(r)\big)' + \lambda r^{-2} f(r)\big|^2     \no \\
&\quad= \int_{r_1}^{r_2} dr \, r^{\alpha+n-1} |f''(r)|^{2} + [2\lambda + (n-1)(1-\alpha)]\int_{r_1}^{r_2} dr \, r^{\alpha+n-3} |f'(r)|^{2}     \no \\
&\qquad + \lambda[\lambda + (\alpha+n-4)(2-\alpha)] \int_{r_1}^{r_2} dr \, r^{\alpha+n-5} |f(r)|^{2}     \no \\ 
& \quad \geq \big[4^{-1} (\alpha+n-2)^2 + [\pi/\ln(r_2/r_1)]^2]\big] \big[4^{-1} (\alpha+n-4)^2 
+ [\pi/\ln(r_2/r_1)]^2\big] 
\no \\
& \qquad \times \int_{r_1}^{r_2} dr \, r^{\alpha+n-5} |f(r)|^2    \no \\
& \qquad + [2 \lambda + (n-1)(1-\alpha)] \big[4^{-1} (\alpha+n-4)^2 + [\pi/\ln(r_2/r_1)]^2\big]     \no \\
& \qquad \quad  \times \int_{r_1}^{r_2} dr \, r^{\alpha+n-5} |f(r)|^2 
+ \lambda[\lambda + (\alpha + n-4)(2-\alpha)] \int_{r_1}^{r_2} dr \, r^{\alpha+n-5} |f(r)|^2   \no \\
& \quad = \bigg\{\big(\wti \gamma_{n,\alpha} + \lambda\big)^2 + \Big(2^{-1} \big[(\alpha-2)^2 + (n-2)^2\big] + 2 \lambda\Big) [\pi/\ln(r_2/r_1)]^2     \no \\ 
& \qquad \;\;\, + [\pi/\ln(r_2/r_1)]^4\bigg\} \int_{r_1}^{r_2} dr \, r^{\alpha+n-5} |f(r)|^{2}. 
\end{align}
\end{proof}

The power-weighted Rellich inequality on an annulus then reads as follows:

\begin{theorem} \lb{t4.5}
Let $\alpha \in \bbR$ and $f \in C_{0}^{\infty}(A_n(r_1,r_2))$. Then
\begin{align}
& \int_{A_n(r_1,r_2)} d^n x \, |x|^{\alpha} |(-\Delta f)(x)|^2    \no \\
& \quad \geq \Big\{\wti \beta_{n,\alpha} + 2^{-1} \big[(\alpha-2)^2 + (n-2)^2\big] [\pi/\ln(r_2/r_1)]^2 
+ [\pi/\ln(r_2/r_1)]^4\Big\}    \no \\
& \qquad \times \int_{A_n(r_1,r_2)} d^n x \, |x|^{\alpha-4} |f(x)|^2,     \lb{4.15}
\end{align}
where
\begin{equation}
\wti \beta_{n,\alpha} = \min \Big\{\big[\wti \gamma_{n,\alpha} + \lambda_j\big]^2 \, \Big| \, j \in \bbN_0\Big\}, 
\quad \alpha \in \bbR. 
\lb{4.16} 
\end{equation}
\end{theorem}
\begin{proof}
It suffices to prove \eqref{4.15} for real-valued $f$ only. By Lemmas \ref{l4.2}\,$(ii)$ and \ref{l4.4} one obtains 
\begin{align}
& \int_{A_n(r_1,r_2)} d^n x \, |x|^{\alpha} |(-\Delta f)(x)|^2    \no \\
& \quad = \int_{r_1}^{r_2} r^{n-1} dr \int_{\bbS^{n-1}} d^{n-1} \omega(\theta)   \no \\
& \qquad \times \sum_{j \in \bbN_0} \sum_{\ell=1}^{m(\lambda_j)} r^{\alpha} \big|- r^{1-n}\big(r^{n-1} F_{f,j,\ell}^{\, \prime}(r)\big)' + \lambda_j^2 r^{-2} F_{f,j,\ell}(r)\big|^2 |\varphi_{j,\ell}(\theta)|^2    \no \\
& \quad = \sum_{j \in \bbN_0} \sum_{\ell=1}^{m(\lambda_j)} \int_{r_1}^{r_2} dr \, r^{\alpha+n-1} 
\big|- r^{1-n}\big(r^{n-1} F_{f,j,\ell}^{\, \prime}(r)\big)' + \lambda_j^2 r^{-2} F_{f,j,\ell}(r)\big|^2     \no \\
& \quad \geq \sum_{j \in \bbN_0} \sum_{\ell=1}^{m(\lambda_j)} \bigg\{\big(\wti \gamma_{n,\alpha} + \lambda_j \big)^2 + \Big(2^{-1} \big[(\alpha-2)^2 + (n-2)^2\big] + 2 \lambda_j\Big) [\pi/\ln(r_2/r_1)]^2    \no \\
& \hspace*{2.55cm} + [\pi/\ln(r_2/r_1)]^4\bigg\} \int_{r_1}^{r_2} dr \, r^{\alpha+n-5} |F_{f,j,\ell}(r)|^{2}    \no \\
& \quad \geq \Big\{\wti \beta_{n,\alpha} + 2^{-1} \big[(\alpha-2)^2 + (n-2)^2\big] [\pi/\ln(r_2/r_1)]^2 
+ [\pi/\ln(r_2/r_1)]^4\Big\}    \no \\
& \qquad \times \sum_{j \in \bbN_0} \sum_{\ell=1}^{m(\lambda_j)} 
\int_{r_1}^{r_2} dr \, r^{\alpha+n-5} |F_{f,j,\ell}(r)|^{2}     \no \\
& \quad = \Big\{\wti \beta_{n,\alpha} + 2^{-1} \big[(\alpha-2)^2 + (n-2)^2\big] [\pi/\ln(r_2/r_1)]^2 
+ [\pi/\ln(r_2/r_1)]^4\Big\}    \no \\
& \qquad \times \sum_{j \in \bbN_0} \sum_{\ell=1}^{m(\lambda_j)} 
\int_{r_1}^{r_2} dr \int_{\bbS^{n-1}} d^{n-1}\omega(\theta)  \, r^{\alpha+n-5} |F_{f,j,\ell}(r)|^{2} 
|\varphi_{j,\ell}(\theta)|^2    \no \\
& \quad = \Big\{\wti \beta_{n,\alpha} + 2^{-1} \big[(\alpha-2)^2 + (n-2)^2\big] [\pi/\ln(r_2/r_1)]^2 
+ [\pi/\ln(r_2/r_1)]^4\Big\}    \no \\
& \qquad \times \int_{A_n(r_1,r_2)} d^n x \, |x|^{\alpha-4} |f(x)|^2. 
\end{align}
\end{proof}

\begin{remark} \lb{r4.6}
We note that as $r_1 \downarrow 0$ and $r_2 \uparrow \infty$ in inequality \eqref{4.15}, one recovers the optimal Rellich inequality of Caldirola and Musina \cite[Theorem~3.1, eqs.~(1.6), (1.7)]{CM12} and \cite{GPPS24}. Also,  keeping $r_2 > 0$ fixed and letting $r_1 \downarrow 0$ in \eqref{4.15}, one recovers the optimal Rellich inequality in \cite[Eq. (A.49) and (A.50)]{GPPS24} for the punctured ball of radius $R > 0$. 
\hfill $\diamond$
\end{remark}

An appropriate iteration procedure then yields the following higher-order power-weighted Rellich inequalities.

\begin{theorem} \lb{t4.7}
Let $m,n \in \bbN$, $n \geq 2$, and $\alpha \in \bbR$. Then 
\begin{align}
\begin{split} 
& \int_{A_n(r_1,r_2)} d^n x \, |x|^{\alpha} \big|\big((-\Delta)^m f\big)(x)\big|^2 \geq 
\Bigg(\prod_{j=1}^m D_{n,r_1,r_2}(\alpha - 4(j-1))\Bigg)    \lb {4.18} \\
& \quad \times \int_{A_n(r_1,r_2)} d^n x \, |x|^{\alpha-4m} |f(x)|^2, 
\quad f \in C_0^{\infty}(A_n(r_1,r_2)),
\end{split} 
\end{align}
where
\begin{equation}
D_{n,r_1,r_2}(\gamma) = \wti \beta_{n,\gamma} + 2^{-1} \big[(\gamma-2)^2 + (n-2)^2\big] [\pi/\ln(r_2/r_1)]^2 
+ [\pi/\ln(r_2/r_1)]^4, \quad \gamma \in \bbR.    \lb{4.19} 
\end{equation}
Moreover,
\begin{align}
\begin{split} 
\int_{\bbR^n} d^n x \, |x|^{\alpha} \big|\big((-\Delta)^m f\big)(x)\big|^2 \geq 
\Bigg(\prod_{j=1}^m \wti \beta_{n,\alpha-4(j-1)}\Bigg) 
\int_{\bbR^n} d^n x \, |x|^{\alpha-4m} |f(x)|^2,&    \lb {4.20} \\
f \in C_0^{\infty}(\bbR^n\backslash\{0\}).&
\end{split} 
\end{align}
\end{theorem}
\begin{proof}
We will employ induction on $m$. Clearly, Theorem \ref{t4.5} implies that \eqref{4.18} holds for $m=1$. Assuming \eqref{4.18} holds for some $m_0 \in \bbN$, one obtains upon applying \eqref{4.18} to $-\Delta f$, 
\begin{align}
\begin{split} 
& \int_{A_n(r_1,r_2)} d^n x \, |x|^{\alpha} \big|\big((-\Delta)^{m_0+1} f\big)(x)\big|^2 \geq 
\Bigg(\prod_{j=1}^{m_0} D_{n,r_1,r_2}(\alpha - 4(j-1))\Bigg)    \lb {4.21} \\
& \quad \times \int_{A_n(r_1,r_2)} d^n x \, |x|^{\alpha-4m_0} |(-\Delta f)(x)|^2, 
\quad f \in C_0^{\infty}(A_n(r_1,r_2)).
\end{split} 
\end{align}
Employing Theorem \ref{t4.5} to the integral on the right-hand side of \eqref{4.21} results in   
\begin{align}
& \int_{A_n(r_1,r_2)} d^n x \, |x|^{\alpha} \big|\big((-\Delta)^{m_0+1} f\big)(x)\big|^2 \geq 
\Bigg(\prod_{j=1}^{m_0} D_{n,r_1,r_2}(\alpha - 4(j-1))\Bigg)   \no \\
& \qquad \times D_{n,r_1,r_2}(\alpha-4m_0) \int_{A_n(r_1,r_2)} d^n x \, |x|^{\alpha-4m_0-4} |f(x)|^2 \no \\
& \quad = \Bigg(\prod_{j=1}^{m_0+1} D_{n,r_1,r_2}(\alpha - 4(j-1))\Bigg)
 \int_{A_n(r_1,r_2)} d^n x \, |x|^{\alpha-4(m_0+1)} |f(x)|^2,    \lb{4.22} \\
& \hspace*{7.7cm} f \in C_0^{\infty}(A_n(r_1,r_2)),    \no
\end{align}
and hence in \eqref{4.18} for $m_0+1$.

Inequality \eqref{4.20} follows upon observing that 
\begin{equation}
\lim_{r_1 \downarrow 0, \, r_2 \uparrow \infty} D(\alpha - 4(j-1)) = \wti \beta_{n,\alpha - 4(j-1)}, \quad 
\alpha \in \bbR, \; j=1,\dots,m.     \lb{4.23} 
\end{equation}
\end{proof}

Similarly, an iteration procedure also yields the following sequence of power-weighted Hardy--Rellich-type inequalities. 

\begin{theorem} \lb{t4.8}
Let $m,n \in \bbN$, $n \geq 2$, and $\alpha \in \bbR$. Then 
\begin{align}
& \int_{A_n(r_1,r_2)} d^n x \, |x|^{\alpha} \big|\big(\nabla (-\Delta)^m f\big)(x)\big|^2     \no \\
& \quad \geq 
\big\{4^{-1} (\alpha+n-2)^2 + [\pi/\ln(r_2/r_1)]^2\big\} 
\Bigg(\prod_{j=1}^m D_{n,r_1,r_2}(\alpha - 2 - 4(j-1))\Bigg)    \lb {4.24} \\
& \qquad \times \int_{A_n(r_1,r_2)} d^n x \, |x|^{\alpha-2-4m} |f(x)|^2, 
\quad f \in C_0^{\infty}(A_n(r_1,r_2)).     \no 
\end{align}
Moreover,
\begin{align}
\begin{split} 
& \int_{\bbR^n} d^n x \, |x|^{\alpha} \big|\big(\nabla (-\Delta)^m f\big)(x)\big|^2     \\
& \quad \geq 
4^{-1} (\alpha+n-2)^2 \Bigg(\prod_{j=1}^m \wti \beta_{n,\alpha-2-4(j-1)}\Bigg)  
\int_{\bbR^n} d^n x \, |x|^{\alpha-2-4m} |f(x)|^2,    \lb {4.25} \\
& \hspace*{7.8cm} f \in C_0^{\infty}(\bbR^n\backslash\{0\}). 
\end{split} 
\end{align}
\end{theorem} 
\begin{proof}
Once more we will employ induction on $m$. Applying Theorems \ref{t4.1} and \ref{t4.5} one obtains for $m=1$,
\begin{align}
& \int_{A_n(r_1,r_2)} d^n x \, |x|^{\alpha} |(\nabla (-\Delta f)(x)|^2     \no \\
& \quad \geq 
\big\{4^{-1} (\alpha+n-2)^2 + [\pi/\ln(r_2/r_1)]^2\big\} 
\int_{A_n(r_1,r_2)} d^n x \, |x|^{\alpha-2} |(-\Delta f)(x)|^2    \no \\
& \quad \geq \big\{4^{-1} (\alpha+n-2)^2 + [\pi/\ln(r_2/r_1)]^2\big\} D_{n,r_1,r_2}(\alpha-2)   \no\\
& \qquad \times \int_{A_n(r_1,r_2)} d^n x \, |x|^{\alpha-6} |f(x)|^2, 
\quad f \in C_0^{\infty}(A_n(r_1,r_2)),   \lb{4.26}
\end{align}
proving \eqref{4.24} for $m=1$. Next, suppose that \eqref{4.24} holds for some $m_0 \in \bbN$. Then, applying \eqref{4.24} to $-\Delta f$ yields 
\begin{align}
& \int_{A_n(r_1,r_2)} d^n x \, |x|^{\alpha} \big|\big(\nabla (-\Delta)^{m_0+1} f\big)(x)\big|^2     \no \\
& \quad \geq 
\big\{4^{-1} (\alpha+n-2)^2 + [\pi/\ln(r_2/r_1)]^2\big\} 
\Bigg(\prod_{j=1}^{m_0} D_{n,r_1,r_2}(\alpha - 2 - 4(j-1))\Bigg)   \no \\
& \qquad \times \int_{A_n(r_1,r_2)} d^n x \, |x|^{\alpha-2-4m_0} |(-\Delta f)(x)|^2, 
\quad f \in C_0^{\infty}(A_n(r_1,r_2)).    \lb{4.27}  
\end{align}
An application of Theorem \ref{t4.5} to the last integral in \eqref{4.27} finally implies
\begin{align}
& \int_{A_n(r_1,r_2)} d^n x \, |x|^{\alpha} \big|\big(\nabla (-\Delta)^{m_0+1} f\big)(x)\big|^2     \no \\
& \quad \geq 
\big\{4^{-1} (\alpha+n-2)^2 + [\pi/\ln(r_2/r_1)]^2\big\} 
\Bigg(\prod_{j=1}^{m_0} D_{n,r_1,r_2}(\alpha - 2 - 4(j-1))\Bigg)  \no \\
& \qquad \times D_{n,r_1,r_2}(\alpha-2-4m_0) \int_{A_n(r_1,r_2)} d^n x \, |x|^{\alpha-2-4m_0-4} |f(x)|^2,   \no \\
& \quad = 
\big\{4^{-1} (\alpha+n-2)^2 + [\pi/\ln(r_2/r_1)]^2\big\}
\Bigg(\prod_{j=1}^{m_0+1} D_{n,r_1,r_2}(\alpha - 2 - 4(j-1))\Bigg)  \no \\
& \qquad \times \int_{A_n(r_1,r_2)} d^n x \, |x|^{\alpha-2-4(m_0+1)} |f(x)|^2, 
\quad f \in C_0^{\infty}(A_n(r_1,r_2)),     
\end{align}
and hence \eqref{4.24} for $m_0+1$. 

Inequality \eqref{4.25} follows upon observing that 
\begin{equation}
\lim_{r_1 \downarrow 0, \, r_2 \uparrow \infty} D(\alpha - 2 - 4(j-1)) = \wti \beta_{n,\alpha - 2 - 4(j-1)}, \quad 
\alpha \in \bbR, \; j=1,\dots,m.     \lb{4.29} 
\end{equation}
\end{proof}

Together, \eqref{4.18} and \eqref{4.24} constitute what we dubbed the Birman--Hardy--Rellich-type sequence of power-weighted Hardy--Rellich-type inequalities on annuli \cite{Ba07}, \cite{DH98}, \cite[Sect.~6.4]{GM13} (see also \cite{Bi66} for the original one-dimensional sequence of Hardy--Rellich-type inequalities). 

\begin{corollary} \lb{c4.9} 
Let $m,n \in \bbN$, $n \geq 2$, and $\alpha \in \bbR$. Then the following items $(i)$ and $(ii)$ hold. \\[1mm]
$(i)$ If $\alpha \in [4m-n, n]$, then 
\begin{align}
\begin{split} 
\int_{\bbR^n} d^n x \, |x|^{\alpha} \big|\big((-\Delta)^m f\big)(x)\big|^2 \geq 
\Bigg(\prod_{j=1}^m \wti \gamma_{n,\alpha-4(j-1)}^2\Bigg) \int_{\bbR^n} d^n x \, |x|^{\alpha-4m} |f(x)|^2,&    \lb {4.30} \\
f \in C_0^{\infty}(\bbR^n\backslash\{0\}).&
\end{split} 
\end{align}
$(ii)$ If $\alpha \in [2+4m-n,n+2]$, then 
\begin{align}
\begin{split} 
& \int_{\bbR^n} d^n x \, |x|^{\alpha} \big|\big(\nabla (-\Delta)^m f\big)(x)\big|^2     \\
& \quad \geq 
4^{-1} (\alpha+n-2)^2 \Bigg(\prod_{j=1}^m \wti \gamma_{n,\alpha-2-4(j-1)}^2\Bigg) 
\int_{\bbR^n} d^n x \, |x|^{\alpha-2-4m} |f(x)|^2,    \lb {4.31} \\
& \hspace*{7.8cm} f \in C_0^{\infty}(\bbR^n\backslash\{0\}). 
\end{split} 
\end{align}
\end{corollary} 
\begin{proof}
We start by noting that according to \eqref{4.13}, $\wti \gamma_{n,\alpha} \in [0,\infty)$ if and only if 
$\alpha \in [4-n,n]$. Thus, by \eqref{4.16}, 
\begin{equation}
\wti \beta_{n,\alpha} = \wti \gamma_{n,\alpha}^2, \quad \alpha \in [4-n,n].    \lb{4.32}
\end{equation}
$(i)$ By \eqref{4.20}, in order to prove \eqref{4.30} it suffices to show that 
\begin{equation}
\wti \beta_{n,\alpha-4(j-1)} = \wti \gamma_{n,\alpha-4(j-1)}^2, \quad \alpha \in [4m-n,n], \; j=1,\dots,m.    \lb{4.33} 
\end{equation}
By \eqref{4.32}, one obtains \eqref{4.33} if 
\begin{equation}
\alpha - 4(j-1) \in [4-n,n], \quad j=1,\dots,m.    \lb{4.34}
\end{equation}
But \eqref{4.34} follows from the assumption $\alpha \in [4m-n,n]$. \\[1mm]
$(ii)$ By \eqref{4.25}, in order to prove \eqref{4.31} it suffices to show that 
\begin{equation}
\wti \beta_{n,\alpha-2-4(j-1)} = \wti \gamma_{n,\alpha-2-4(j-1)}^2, \quad 
\alpha \in [2+4m-n,n+2], \;j=1,\dots,m.    \lb{4.35} 
\end{equation}
By \eqref{4.32}, one obtains \eqref{4.35} if 
\begin{equation}
\alpha -2 - 4(j-1) \in [4-n,n], \quad j=1,\dots,m.    \lb{4.36}
\end{equation}
However, \eqref{4.36} follows from the assumption $\alpha \in [2+ 4m - n,n+2]$. 
\end{proof}

\begin{remark} \lb{r4.10}
$(i)$ If $\alpha \in (4m-n,2)$ (resp., $\alpha \in (4m-n,0]$), then \eqref{4.30} was proven by Davies and Hinz \cite[Theorem~12]{DH98} (resp., by Barbatis \cite[Theorem~1]{Ba07}). \\
$(ii)$ If $\alpha \in (2+4m-n,4]$ (resp., if $\alpha \in (2+4m-n,0]$), then \eqref{4.31} was proven by Davies and Hinz \cite[Theorem~13]{DH98} (resp., by Barbatis \cite[Theorem~1]{Ba07}). 
\end{remark} 

\medskip

\noindent
{\bf Acknowledgments.} We are indebted to the anonymous referee for a very careful reading of our manuscript. 


\end{document}